\documentclass[a4paper]{amsart}
\usepackage{amsmath}%
\usepackage{amsfonts}%
\usepackage{amssymb}%
\usepackage{graphicx}
\usepackage{amsaddr}
\newtheorem{thm}{Theorem}[section]
\theoremstyle{plain}
\newtheorem{cor}[thm]{Corollary}
\newtheorem{deff}[thm]{Definition}
\newtheorem{lem}[thm]{Lemma}
\newtheorem{prop}[thm]{Proposition}
\newtheorem{exa}[thm]{Example}
\newtheorem{note}[thm]{Note}
\long\def\symbolfootnote[#1]#2{\begingroup\def\thefootnote{\fnsymbol{footnote}}\footnote[#1]{#2}\endgroup}

\begin{document}
\title[Wijsman Topology]{On normality of the Wijsman Topology}
\author{L\!'ubica Hol\'a and Branislav Novotn\'y}
\address{Mathematical Institute, Slovak Academy of Sciences, \v Stef\'anikova 49, SK-814 73 Bratislava, Slovakia}
\email{lubica.hola@mat.savba.sk, branislav.novotny@mat.savba.sk}
\date{}
\begin{abstract}
Let $(X,\rho)$ be a metric space and $(CL(X),W_\rho)$ be the hyperspace of all nonempty closed subsets of $X$ equipped with the Wijsman topology. The Wijsman topology is one of the most important classical hyperspace topologies. We give a partial answer to a question posed in \cite{maio} whether the normality of $(CL(X),W_\rho)$ is equivalent to its metrizability. If $(X,\rho)$ is a linear metric space, then $(CL(X),W_\rho)$ is normal if and only if $(CL(X),W_\rho)$ is metrizable. Some further results concerning normality of the Wijsman topology on $CL(X)$ are also proved.
\end{abstract}
\maketitle
\symbolfootnote[0]{Both authors were supported by VEGA 2/0047/10.}
\section{Introduction}


Let $(X,\rho)$ be a metric space and $(CL(X),W_\rho)$ be the hyperspace of all nonempty closed subsets of $X$ equipped with the Wijsman topology.
The Wijsman topology is now considered as a classical one. It is one of the most important hyperspace topologies. The Wijsman topology is finer than the Fell topology and weaker than the Vietoris topology and the Husdorff metric topology.

The Wijsman topology  $W_\rho$ is the weak topology (initial topology) determined by all distance functionals $\rho(x,\cdot):CL(X)\to [0,\infty)$, where $\rho(x,A)=inf\{\rho(x,a);a\in A\}$ \cite{beer}; i.e. $A_\lambda\to A$ iff $\rho(x,A_\lambda)\to\rho(x,A)$ for all $x\in X$.

This topology, or more precisely convergence, was introduced in 1966 by R. Wijsman for closed convex sets in $\mathbf{R}^n$ \cite{wij}. Since then the Wijsman topology was studied by many authors and found many applications; see e.g. \cite{beer2}, \cite{costantini}, \cite{holluc}, \cite{maio}, \cite{zsilinszky}.
In \cite{belelena} there is proved that for a metrizable space $(X,\tau)$, the Vietoris topology is the smallest topology containing all Wijsman topologies determined by metrics compatible with $\tau$ and for a metric space $(X,\rho)$, the proximal topology is the smallest topology containing all Wijsman topologies determined by metrics which are uniformly equivalent to $\rho$.
As a basic reference for the Wijsman topology we recommend \cite{beer}.


In this paper we study the normality and cardinal invariants of the Wijsman topology. Notice that the normality of the Vietoris topology on $CL(X)$ was studied by Keesling in \cite{keesling1}, \cite{keesling2}, but it was solved by Velichko in \cite{vel}; he proved that it is equivalent to the compactness of $X$. Hol\'a, Levi and Pelant proved in \cite{hollevpel} that the normality of the Fell topology on $CL(X)$ is equivalent to the local compactness and Lindel\"ofness of $X$.

In our paper we give a partial answer to a question posed in \cite{maio}:
\emph{It is known that if $(X,\rho)$ is a separable metric space, then $(CL(X),W_\rho)$ is metrizable and so paracompact and normal. Is the opposite true? Is $(CL(X),W_\rho)$ normal if and only if $(CL(X),W_\rho)$ is metrizable?}


\section{Preliminaries}

Let $s,e,c,d,nw,w,\psi w,\pi,\chi,\psi,\pi\chi,t,L$ be cardinal invariants of a topological space: spread, extent, cellularity, density, netweight, weight, pseudo weight, $\pi-$weight, character, pseudocharacter, $\pi-$character, tightness, and Lindel\"of number respectively, as defined in \cite{juhasz} and \cite{engelking}. All are greater or equal to $\aleph_0$.

Consider a topological space $(X,\tau)$ and $x\in X$. Arguments in cardinal functions will be denoted as 
 $f(x,X,\tau)$. We will omit specification of the topology $\tau$ when it will not be confusing. Points will be specified only for a point specific cardinal functions ($\psi,\chi,t,...$) and  in this case we have $f(X)=sup\{f(x,X);x\in X\}$. Every cardinal function has also a hereditary version $hf(X)=sup\{f(Y);Y$ is a subspace of $X\}$.

Now define some other cardinal functions. $|X|$ denotes the cardinality of $X$, $card(X)=\aleph_0+|X|$, $o(X)=\aleph_0+|\tau|$, the diagonal degree by $\Delta(X)=\aleph_0+min\{|\mathcal{G}|; \mathcal{G}$ is a family of open sets in $X\times X$ with $\cap\mathcal{G}$ equal to the diagonal in $X\times X\}$ \cite{hodel}.
For a Tychonoff space define the uniform weight by $u(X)=\aleph_0+min\{|\mathcal{W}|; \mathcal{W}$ is a base for compatible uniformity$\}$ \cite{engelking} and the weak weight by $ww(X)=min\{w(Y);$ there is a continuous bijection from $X$ onto a Tychonoff space $Y\}$ \cite{mccoy}.

A metric space $(X,\rho)$ is $\epsilon-$discrete iff for any distinct $x,y\in X$ holds $\rho(x,y)\geq\epsilon$; and it is uniformly discrete iff it is $\epsilon-$discrete for some $\epsilon>0$. In a metric space denote by $S(x,\epsilon)$ ($B(x,\epsilon)$) an open (closed) ball with the radius $\epsilon$ and the center $x$. $S(M,\epsilon)=\bigcup_{x\in M}S(x,\epsilon)$ and $B(M,\epsilon)=\bigcup_{x\in M}B(x,\epsilon)$. If we need to specify the metric $\rho$, we will write $S_\rho(x,\alpha), B_\rho(x,\alpha), S_\rho(M,\epsilon)$ and $B_\rho(M,\epsilon)$.

\begin{note}[{\cite[2.1.]{juhasz}, \cite[Fig. 1.]{hodel}, \cite[8.5.17.]{engelking}, \cite[IV.9.16.]{mccoy}}]\label{notecard}
	The diagram in the \emph{Figure} \ref{figcard} shows relations among cardinal invariants on a Tychonoff space. (Without specifying a point i.e. only f(X).) Functions connected by a line are comparable and the upper one is greater than or equal to the lower one. This is true also for a $T_1$ space, but one should omit those in boxes.
\end{note}
\begin{figure}[h]
	\includegraphics[width=0.7\textwidth,angle=270]{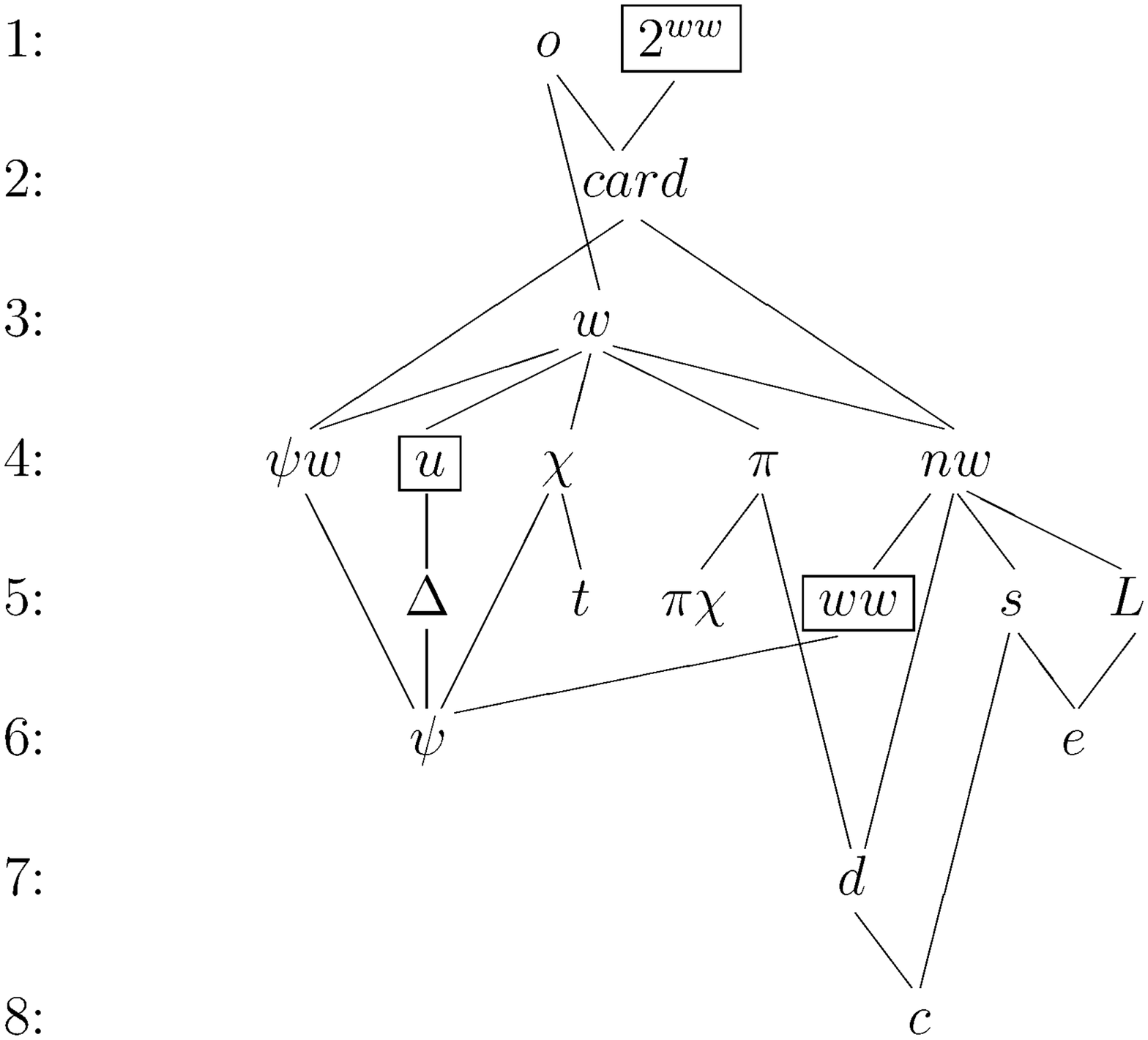}
\caption{}
\label{figcard}
\end{figure}

\begin{note}\label{notemetric}
For a metric space $X$ is $w(X)=nw(X)=\pi(X)=L(X)=s(X)=e(X)=d(X)=c(X)=sup\{card(D); D$ is a uniformly discrete subspace of $X\}$.
\end{note}
\section{Cardinal Invariants of the Wijsman Topology}

We will work on a metric space $(X,\rho)$ and its hyperspace $(CL(X),W_\rho)$.

 If $X$ is only a topological space we can consider the Vietoris topology $V$ on $CL(X)$. For $U\subset X$ denote $U^-=\{A\in CL(X);A\cap U\not=\emptyset\}$ and $U^+=\{A\in CL(X);A\subset U\}$. The family $\{U^-; U\text{ is open}\}$ ($\{U^+; U\text{ is open}\}$) is a subbase of $V^-$, the lower part of the Vietoris topology ($V^+$, the upper part of the Vietoris topology $V^+$). Together they form a subbase of $V$. Analogically we have the lower and upper parts of the Wijsman topology $W_\rho^-, W_\rho^+$ with the following subbases
\[\{S^-(x,\alpha); x\in E, \alpha\in Q^+\}\]
\[\{S^+(x,\alpha); x\in E, \alpha\in Q^+\}\]
respectively; where
\[S^-(x,\alpha)=\{A\in CL(X);\rho(x,A)<\alpha\}=S(x,\alpha)^-,\]
\[S^+(x,\alpha)=\{A\in CL(X);\rho(x,A)>\alpha\}=\bigcup_{\beta > \alpha}B(x,\beta)^{C+},\]
$E$ is a dense subset of $X$ and $Q^+$ is the set of all positive rationals. If need to specify the metric $\rho$, we will use notations $S^-_\rho(x,\alpha)$ and $S^+_\rho(x,\alpha)$.

We also have a subbase for a natural uniformity compatible with $W_\rho$: \[\{W_{x,n}; x\in E, n\in\omega\},\] where $W_{x,n}=\{(A,B)\in CL(X)^2;|\rho(x,A)-\rho(x,B)|<1/n\}$.
\bigskip

We will suppose that $(X,\rho)$ is a metric space and $CL(X)$ is equipped with $W_\rho$ if not stated explicitly otherwise. Notice that if $(X,\rho)$ is a metric space, then $W_\rho^- = V^-$ on $CL(X)$.
\bigskip

We immediately obtain:
\begin{prop}
	$w(CL(X))\leq d(X)$.
\end{prop}
We will derive now some upper estimates of $d(X)$ from more general results about the Vietoris topology.
\begin{lem}
	Let $X$ be a topological space. Let $\mathcal G=\{G_\lambda\subset X;\lambda\in\Lambda\}$ be a family of nonempty open sets. Put $\mathcal U=\{\cap_{\lambda\in I} G_\lambda^-;\textrm{finite}\ I\subset\Lambda\}$.
The following are equivalent:
	\begin{enumerate}
		\item	$\mathcal G$ is a $\pi-$base of $X$,
		\item $\mathcal U$ is a $\pi-$base of $(CL(X),V^-)$,
		\item $\mathcal U$ is a local $\pi-$base at $X$ in $(CL(X),V^-)$,
		\item $\cap\mathcal U=\{X\}$; i.e. $\mathcal U$ is a local pseudo base at $X$ in $(CL(X),V^-)$.
	\end{enumerate}
\end{lem}
\begin{proof}
	$(1)\Rightarrow(2):$ It follows from the fact that if $G_\lambda\subset U$, then $G_\lambda^-\subset U^-$.
	
	$(2)\Rightarrow(3):$ It follows from the fact that every open set in $V^-$ contains $X$.
	
	$(3)\Rightarrow(4):$ Clearly $X\in\cap\mathcal U$. Consider closed $A\not=X$. So $(X\setminus A)^-\supset\mathfrak U$ for some $\mathfrak U\in\mathcal U$. Then $A\not\in\mathfrak U$ and so $A\not\in\cap\mathcal U$.
	
	$(4)\Rightarrow(1):$ Suppose that $\mathcal G$ is not a $\pi-$base, then there is a closed set $A\subset X$ which meets every $G_\lambda\in\mathcal G$. Then clearly $A\in\cap\mathcal U$.
\end{proof}
W.l.o.g. we can consider for a pseudo (resp. $\pi-$) base in $(CL(X),V^-)$ only systems of the form $\mathcal U$ from the previous lemma. We have the following theorem as a direct corollary.
\begin{thm}
	Let  $X$ be a topological space. Then $\psi(X,CL(X),V^-)=\pi(X)=\pi\chi(X,CL(X),V^-)=\pi\chi(CL(X),V^-)$.
\end{thm}
Note that every open neighborhood of $X$ in $W_\rho$ is a member of $V^-$. We have the following corollary.
\begin{cor}
	$d(X)\leq \min\{\psi(CL(X)),\pi\chi(CL(X))\}$.
\end{cor}

For $A\subset X$ denote by $\mathcal{F}(A)$ the set of all finite subsets of $A$. The following lemma is well-known; see e.g. \cite{mizokami}.
\begin{lem}\label{dense}
	If $E$ is dense in $X$, then $\mathcal{F}(E)$ is dense in $(CL(X),V)$.
\end{lem}

\begin{thm}
	$d(X)\leq t(X,CL(X),V^-)\leq t(CL(X),W_\rho)$.
\end{thm}
\begin{proof}
	The second inequality follows from the fact that every $W_\rho$ neighborhood of $X$ belongs to $V^-$. The first inequality is simmilar to \cite[2.4.]{cosholvit}. Since $X\in\overline{\mathcal{F}(X)}$ there is $\mathcal{A}\subset\mathcal{F}(X)$ such that $|\mathcal{A}|\leq t(X,CL(X),V^-)$ and $X\in\overline{\mathcal{A}}$. For every open $U\subset X$, $U^-$ is a neighborhood of $X$ and therefore contains some $F\in\mathcal{A}$, i.e. $U$ meets $F$ and then obviously meets $\cup\mathcal{A}$, which is therefore dense in $X$.
\end{proof}

Since $X$ is a closed subset of $CL(X)$ we have: $e(CL(X))\geq e(X)=d(X)$. And finally from previous results, trivial inequalities from Note \ref{notecard} and the fact that $hw=w$ we have the following theorem.
\begin{thm}
	$d(X)=f(CL(X))=hf(CL(X))$, where $f$ is any function from $\psi,\psi w,\pi\chi,\pi,nw,t,L,w,\Delta,u,e,s,ww$.
\end{thm}

Those are all functions from Figure \ref{figcard}, lines 3--6.


\section{Density and Celularity}

The following proposition is a direct corollary of Lemma \ref{dense}.
\begin{prop}\label{density}
	$d(CL(X))\leq d(X)$.
\end{prop}
Then $d(X)\geq hd(CL(X))\geq hc(CL(X))=s(CL(X))=d(X)$ i.e.
\begin{cor}
	$hd(CL(X))=d(X)$.
\end{cor}
	For a cardinal number  $\mathfrak n$ define $\log(\mathfrak n)=min\{\mathfrak m;\mathfrak n\leq 2^\mathfrak{m}\}$ \cite{mccoy}. Observe that $\log(2^\mathfrak{n})\leq\mathfrak n\leq2^{\log(\mathfrak n)}$.
\begin{lem}\label{lowbound}
	$\log(d(X))\leq ww(X)\leq d(CL(X))$.
\end{lem}
\begin{proof}
	The first inequality follows from Note \ref{notecard} and for the second suppose $\mathcal D$ is dense in $CL(X)$ with the cardinality $d(CL(X))$. Define $H:X\to R^\mathcal{D}$ by $(\pi_A\circ H)(x)=d(x,A)$ for each $x\in X$ and $A\in\mathcal D$. $H$ is a continuous injection, thus $ww(X)\leq w(H(X))\leq |\mathcal D|\leq d(CL(X))$.
\end{proof}
	We will provide an example where the inequality from Corollary \ref{density} is sharp.
\begin{exa}
	There is $X$ with $d(CL(X))<d(X)$.
\end{exa}
\begin{proof}
	Let $(X,\mu)$ be a separable metric space with $|X|>\aleph_0$ and let $\rho$ be the $0-1$ metric on $X$. Let $\{x_i;i\in\omega\}$ be a dense set in $(X,\mu)$. Put $\mathcal H=\{B_\mu(x_i,1/j);i,j\in\omega\}$. Let $\mathcal L$ be a family of all finite unions of elements from $\mathcal H$. Then $\mathcal L$ is dense in $(CL(X),W_\rho)$.
\end{proof}
	The result about discrete metric spaces can be generalized.
\begin{exa}\label{hodelex}
	Let $X$ be a discrete metric space with the $0-1$ metric. Then $d(CL(X))=\log(d(X))$ and $c(CL(X))=\aleph_0$.
\end{exa}
\begin{proof}
	Let $2^X$ be a space of functions from $X$ to $\{0,1\}$ equipped with the topology of pointwise convergence. Let $\mathbf 1\in 2^X$ be a constant function with value $1$. $CL(X)$ is homeomorphic to $2^X\setminus\{\mathbf 1\}$ thus $d(CL(X))=d(2^X\setminus\{\mathbf 1\})=d(2^X)$ and $c(CL(X))=c(2^X\setminus\{\mathbf 1\})=c(2^X)$. From \cite[11.8.]{hodel} we have that $c(2^X)=\aleph_0$ and $d(2^X)=\log(|X|)$.
\end{proof}

	So we have an example of a space where density and cellularity reach their respective lower bounds. Now we will construct one where they will reach their upper bounds. We will use the following metric.
\begin{deff}
	Let $X$ be a nonempty set, $M\subset X$ and $|M|$ is either infinite or even. For every $x\in M$ define its reflection $x'\in M$ such that $x'\not=x$ and $(x')'=x$. Define a metric $\rho_M$ on $X$ by $\rho_M(x,x)=0$ for $x\in X$, $\rho_M(x,x')=2$ for $x\in M$ and $\rho_M(x,y)=1$ otherwise. Such $\rho_M$ will be called  $M-$metric. Put $\mathcal S_M=\{\{\{x\}\};x\in M\}$.
\end{deff}

	Note that if $X$ is a set, $M\subset X$, $\rho$ is  $0-1$ metric on $X$ and $\rho_M$ is  $M-$metric on $X$; then $W_\rho\cup\mathcal S_M$ is a base of $W_{\rho_M}$.
\begin{exa}
	Let $X$ be a set with $|X|=\kappa\geq\aleph_0$ and $M\subset X$ with $|M|=\mathfrak m\geq\aleph_0$. Equip $X$ with  $M-$metric $\rho_M$. Then $c(CL(X))=\mathfrak m$ and $d(CL(X))=\log\kappa+\mathfrak m$.
	
	So we can take $M$ such that $\mathfrak m=\kappa$ (e.g. $M=X$) to obtain $c(CL(X))=d(CL(X))=d(X)$ or we can take $\mathfrak m>\aleph_0$ and $\kappa=2^{2^\mathfrak m}$ to obtain $\aleph_0<c(CL(X))<d(CL(X))<d(X)$.
\end{exa}
\begin{proof}
	Let $\rho$ be  $0-1$ metric and take $W_\rho\cup\mathcal S_M$ as a base of $W_{\rho_M}$. Suppose that $\mathcal U$ is a cellular system consisting of basic open sets. $W_\rho\cap\mathcal U$ is cellular in $W_\rho$ and by Example \ref{hodelex}  $|W_\rho\cap\mathcal U|\leq\aleph_0$. Since $|\mathcal U\cap\mathcal S_M|\leq\mathfrak m$ we have that $c(CL(X))\leq\mathfrak m$. The reverse inequality is due to the fact that $\mathcal S_M$ is cellular. From Lemma \ref{lowbound} follows that $d(CL(X))\geq\log d(X)=\log\kappa$. Trivially $d(CL(X))\geq|\mathcal S_M|=\mathfrak m$ and so $d(CL(X))\geq\log\kappa+\mathfrak m$. For the reverse inequality take $\mathcal D$ a dense subset of $(CL(X),W_\rho)$ with $|\mathcal D|=\log\kappa$. The set $\mathcal D\cup(\bigcup\mathcal S_M)$ is dense in $(CL(X),W_{\rho_M})$.
\end{proof}



\section{Some Results about Normality of the Wijsman Topology}

In \cite[Problem I]{maio} the following question about normality of the Wijsman topology is posed: \emph{It is known that if $(X,\rho)$ is a separable metric space, then $(CL(X),W_\rho)$ is metrizable and so paracompact and normal. Is the opposite true? Is $(CL(X),W_\rho)$ normal if and only if $(CL(X),W_\rho)$ is metrizable?}

We have found several classes of metric spaces for which this is true. And we have also an answer for a weaker question. Suppose $X$ is metrizable. If $X$ is separable then for every compatible metric $\rho$, $(CL(X),W_\rho)$ is metrizable and thus normal. Is the opposite true? If for every compatible metric $\rho$, $(CL(X),W_\rho)$ is normal, does $X$ have  to be separable?

\bigskip

Let us start with a result, which connects this section with the previous one. Note that a metric space is generalized compact (GK) iff for every closed discrete subspace $D\subset X$ we have  $|D|<d(X)$; see \cite[Theorem 7]{barcos}.
\begin{thm}
	If $CL(X)$ is normal then we have the following possibilities:
	\begin{enumerate}
		\item $X$ is not GK. Then $2^{d(CL(X))}=2^{d(X)}$.
		\item $X$ is GK and $d(CL(X))=d(X)$.
		\item $X$ is GK and $d(CL(X))<d(X)$. Then $2^{d(CL(X))}=2^{<d(X)}=sup\{2^\kappa;\kappa<d(X)\}$.
	\end{enumerate}
	Moreover if GCH holds, then always $d(CL(X))=d(X)$.
\end{thm}
\begin{proof}
	Let $D$ be a closed discrete subset of $X$. The number of continuous functions from $D$ to $[0,1]$ is equal to $2^{card(D)}$. The number of continuous functions from $CL(X)$ to $[0,1]$ is less than or equal to $2^{d(CL(X))}$. From Tietze extension theorem we have $2^{card(D)}\leq 2^{d(CL(X))}$.

	$(1)$ If $X$ is not GK, we can take a closed discrete subset $D\subset X$ such that $card(D)=d(X)$ and the rest follows from Proposition \ref{density}.

	$(3)$ If $X$ is GK and $d(CL(X))<d(X)$, then clearly $2^{d(CL(X))}\leq 2^{<d(X)}$. For every $\kappa<d(X)$ we can take a uniformly discrete subset $D\subset X$ such that $\kappa\leq card(D)<d(X)$ by Note \ref{notemetric}; i.e. $2^{\kappa}\leq2^{card(D)}\leq 2^{d(CL(X))}$ and the rest follows.

	Under GCH $2^{card(D)}\leq 2^{d(CL(X))}$ implies $card(D)\leq d(CL(X))$ hence $d(X)=e(X)\leq d(CL(X))\leq d(X)$.
\end{proof}

\bigskip

	Take $\kappa\geq\aleph_0$. Consider a discrete metric space $X$ with $|X|=2^\kappa$ (with $0-1$ metric). By Example \ref{hodelex} we have $d(CL(X))=\log d(X)=\kappa$ and thus $2^{d(X)}=2^{2^\kappa}>2^\kappa=2^{d(CL(X))}$. Since $X$ is not GK then it cannot be normal. This result can be generalized for a discrete metric space with $|X|>\aleph_0$.

\begin{lem}\label{fell}
	Let $\epsilon>0$. Let $(X,\rho)$ be a metric space with $0-\epsilon$ metric $\rho$. If $(CL(X),W_\rho)$ is normal then $X$ is countable.
\end{lem}
\begin{proof}
	The metric space $X$ has nice closed balls, thus by \cite{beer} the Wijsman topology on $CL(X)$ coincides with the Fell topology. By \cite{hollevpel} the normality of the Fell topology on $CL(X)$ implies the Lindel\"ofness of $X$; so we are done.
\end{proof}

	Note that in this case we have that normality of the Wijsman topology is equivalent to metrizability. To apply this result in some other cases, we will use the following lemma.
\begin{lem}\label{subset}
	Let $(X,\rho)$ be a metric space, $Y$ be a closed discrete subset of $X$. Suppose that for every $x\in X\setminus Y$ the following property is fulfilled:

	There is $\eta_x$ and at most one $y_x\in Y$ with $\rho(x,y_x)<\eta_x$, for every other $y\in Y$ holds $\rho(x,y)=\eta_x$.

	Then $(CL(Y),W_{\rho_{|Y}})$ is a closed subspace of $(CL(X),W_\rho)$.
\end{lem}
\begin{proof}
It is well-known that if $Y$ is a closed subset of $X$, then $CL(Y)$ is a closed set in $(CL(X),V^-)$;
thus also in $(CL(X),W_\rho)$.
 Suppose $A_\lambda\in CL(Y)$ converges to $A\in CL(Y)$ with respect to $W_\rho$.  Then for every $x\in X$, $\rho(x,A_\lambda)$ converges to $\rho(x,A)$ and since $Y\subset X$ then $A_\lambda$ converges to $A$ with respect to $W_{\rho_{|Y}}$. Now suppose $A_\lambda\in CL(Y)$ converges to $A\in CL(Y)$ with respect to $W_{\rho_{|Y}}$ and take $x\in X\setminus Y$. We have three possibilities:
	\begin{enumerate}
		\item\label{a}	for every $y\in Y$ holds $\rho(x,y)=\eta_x$,
		\item\label{b}	there is $y_x\in Y$ with $\rho(x,y_x)=\delta<\eta_x$ and $\rho(x,A)<\eta_x$,
		\item\label{c}	there is $y_x\in Y$ with $\rho(x,y_x)=\delta<\eta_x$ and $\rho(x,A)=\eta_x$.
	\end{enumerate}
	
	In the case  $(\ref{a})$ it holds $\rho(x,A_\lambda)=\eta_x\to\eta_x=\rho(x,A)$. In the case  $(\ref{b})$  $y_x\in A$. So eventually $y_x\in A_\lambda$ and hence $\rho(x,A_\lambda)\to\delta=\rho(x,A)$. Finally in the case  $(\ref{c})$ $y_x\not\in A$. Eventually $y_x\not\in A_\lambda$ and hence $\rho(x,A_\lambda)\to\eta_x=\rho(x,A)$.
	
\end{proof}

We can use this lemma in the following example.
\begin{exa}
	Let $\mathfrak m$ be a cardinal number and $J(\mathfrak m)$ be the hedgehog of spininess $\mathfrak m$ (exactly as in \cite[4.1.5]{engelking}). If $CL(J(\mathfrak m))$ equipped with the Wijsman topology is normal, then $\mathfrak m\leq\aleph_0$.
\end{exa}
\begin{proof}
	$J(\mathfrak m)=(I\times S)/\approx$, where $I=[0,1]$, $S$ is an index set with $|S|=\mathfrak m$ and $\approx$ is an equivalence relation; $(x,s)\approx(y,t)$ iff $x=0=y$ or $x=y$ and $s=t$. $J(\mathfrak m)$ is equipped with the metric $\rho$:
	$$\rho((x,s),(y,t))=\left\{\begin{array}{cc}
		|x-y|, &	\text{if }s=t\\
		x+y,	&	\text{if }s\not=t.
	\end{array}\right.$$
	Consider $Y=\{(1,s);s\in S\}$. Since $Y$ fulfills the condition in Lemma \ref{subset} $(CL(Y),W_{\rho_{|Y}})$ is a closed subset of $(CL(J(\mathfrak m)),W_\rho)$ and hence normal. $\rho_{|Y}$ is $0-2$ metric and thus $Y$ is countable and so is $S$.
\end{proof}


For a metric $\rho$ on $X$ and $\eta>0$ denote by $\rho_\eta$ a uniformly equivalent metric defined by $\rho_\eta(x,y)=min\{\rho(x,y),\eta\}$ for $x,y\in X$.
\begin{thm}
	Let $(X,\rho)$ be a metric space. If for every $\epsilon>0$ there is $\eta\in(0,\epsilon)$ such that $(CL(X),W_{\rho_\eta})$ is normal, then $X$ is separable.
\end{thm}
\begin{proof}
	Suppose $X$ is not separable. Then there is an $\epsilon-$discrete set $Y\subset X$ with $|Y|=\aleph_1$. Take $\eta<\epsilon/2$ such that $(CL(X),W_{\rho_\eta})$ is normal. One can easily check the condition in Lemma \ref{subset}, so $(CL(Y),W_{\rho_{\eta|Y}})$ is a closed subset of $(CL(X),W_{\rho_\eta})$ and hence normal. Since $\rho_{\eta|Y}$ is $0-\eta$ metric on $Y$, then $|Y|=\aleph_0$ by Lemma \ref{fell}, which contradicts to the supposition.
\end{proof}
\begin{cor}
	Let $(X,\rho)$ be a metric space. If for every metric $\delta$ (uniformly) equivalent to $\rho$,  $(CL(X),W_\delta)$ is normal, then $X$ is separable.
\end{cor}
\begin{prop}
	Let $(X,\rho)$ be a metric space. The following are equivalent:
	\begin{enumerate}
		\item	Every closed proper ball is totally bounded;
		\item	For every $\eta>0$ $W_\rho=W_{\rho_\eta}$ on $CL(X)$.
	\end{enumerate}
\end{prop}
\begin{proof}
	Naturally $W_\rho^-=W_{\rho_\eta}^-$. For $\alpha<\eta$ holds $S^+_{\rho_\eta}(x,\alpha)=S^+_\rho(x,\alpha)$ and for $\alpha\geq\eta$   $S^+_{\rho_\eta}(x,\alpha)=\emptyset$. Thus $W_\rho^+\supset W_{\rho_\eta}^+$.
	
	$(1)\Rightarrow (2):$ Take $A\in S^+_\rho(x,\alpha)$ so $\rho(x,A)=\beta>\alpha$. Choose $\gamma\in(\alpha,\beta)$ such that $B_\rho(x,\gamma)$ is proper and choose $0<\epsilon<\min\{\eta,(\beta-\gamma)/2\}$ and finite $F\subset X$ such that $S_\rho(F,\epsilon)\supset B_\rho(x,\gamma)$. Then $\rho(F,A)>\epsilon$, i.e. $\rho_\eta(F,A)>\epsilon$ and so $A\in\bigcap_{x\in F}S^+_{\rho_\eta}(x,\epsilon)\subset S^+_\rho(x,\alpha)$, hence $W_\rho^+\subset W_{\rho_\eta}^+$.
	
	$(2)\Rightarrow (1):$ Consider a closed proper ball $B_\rho(x,\alpha)$. For any $\eta>0$, $W_\rho^+\subset W_{\rho_\eta}^+$ so there is finite $F\subset X$ and for every $x\in F$ there is $\beta_x<\eta$ such that $\bigcap_{x\in F}S_{\rho_\eta}^+(x,\beta_x)\subset S_\rho^+(x,\alpha)$. Therefore $$B_\rho(F,\eta)^{C+}\subset \bigcap_{x\in F}S_\rho^+(x,\beta_x)\subset\bigcap_{x\in F}S_{\rho_\eta}^+(x,\beta_x)\subset S_\rho^+(x,\alpha)\subset B_\rho(x,\alpha)^{C+}$$ and hence $B_\rho(x,\alpha)\subset B_\rho(F,\eta)$; i.e. $B_\rho(x,\alpha)$ is totally bounded.
\end{proof}
\begin{cor}
	Let $(X,\rho)$ be a metric space such that every closed proper ball is totally bounded. If $(CL(X),W_\rho)$ is normal, then $X$ is separable.
\end{cor}
	This can be generalized in the following way.
\begin{thm}
	Let $\gamma\geq\omega$ be a regular cardinal number (i.e. $cf(\gamma)=\gamma$). Let $(X,\rho)$ be a metric space such that for every $\epsilon > 0$ each closed proper ball  can be covered by less than $\gamma$ $\epsilon-$balls. If $(CL(X),W_\rho)$ is normal, then $d(X)\leq\gamma$.
\end{thm}
\begin{proof}
	We will prove that for every $\epsilon-$discrete set $E=\{x_\alpha;\alpha<\kappa\}$ we have $cf(\kappa)\leq\gamma$ by contradiction (we can identify $\gamma$ with the first ordinal having the cardinality $\gamma$). Suppose that there is an $\epsilon-$discrete set $E=\{x_\alpha;\alpha<\kappa\}$ with $\gamma<cf(\kappa)$. For every $\alpha < \kappa$ put $D_\alpha=\{x_\beta;\beta\in[\alpha,\kappa)\}$.
	
	1) Observe that if $M$ is an $\epsilon-$discrete set, then for closed proper ball $B(x,\eta)$, $|B(x,\eta)\cap M|<\gamma$.
	
	
	
	2) Put $\mathcal A=\{\{x\};x\in X\}$ and $\mathcal B=\{D_\alpha;\alpha<\kappa\}$. Since $\mathcal A$ and $\mathcal B$ are closed disjoint subsets of $CL(X)$ 
	then there is a continuous function $f:CL(X)\to[0,1]$ with $f(\mathcal A)=\{0\}$ and $f(\mathcal B)=\{1\}$. 
	By a transfinite induction we will construct an increasing $\alpha_\lambda<\kappa$ for $\lambda<\gamma$ such that for $L_\lambda=\{x_{\alpha_\lambda}\}\cup D_{\alpha_{\lambda+1}}$ holds $f(L_\lambda)<1/2$.
	
	Put $\alpha_0=0$.
	Suppose we have $\alpha_\lambda$. Let $\mathcal U$ be a neighborhood of $\{x_{\alpha_\lambda}\}$ such that $f(C)<1/2$ for every $C\in\mathcal U$. There is an open neighborhood $V$ of $x_{\alpha_\lambda}$, finite $F\subset X$ and for every $x\in F$ there is $\eta_x>0$ such that $$\{x_{\alpha_\lambda}\}\in V^-\cap\bigcap_{x\in F}S^+(x,\eta_x)\subset\mathcal U.$$ Since $\gamma<cf(\kappa)$ there must exist $\mu$ such that $\alpha_\lambda<\mu<\kappa$ and $D_\mu\cap\bigcup_{x\in F}B(x,2\eta_x)=\emptyset$. Put $\alpha_{\lambda+1}=\mu$. Then $L_\lambda\in\mathcal U$ and so $f(L_\lambda)<1/2$.
	For a limit ordinal $\lambda$ put $\alpha_\lambda=\sup_{\tau<\lambda}\alpha_\tau$.
	
	3) Put $\beta=\sup_{\lambda<\gamma}\alpha_\lambda$. Since $cf(\kappa)>\gamma$, $\beta<\kappa$. We prove that $L_\lambda\to D_\beta$ for $\lambda\to\gamma$. This is the needed contradiction, because $f(L_\lambda)<1/2$ and $f(D_\beta)=1$.
	 Let $U$ be an open set in $X$ such that $D_\beta\in U^-$. Since for every $\lambda$, $L_\lambda\supset D_\beta$, $L_\lambda\in U^-$. Now let $x\in X$ and $\eta_0>0$ be such that $\rho(x,D_\beta)>\eta_0$. Let $\eta >0$ be such that $\eta_0 < \eta < \rho(x,D_\beta)$.  We claim that there is $\alpha<\beta$ with $B(x,\eta)\cap D_\alpha=\emptyset$. Suppose not. By a transfinite induction we will construct an increasing $\alpha_\lambda<\beta$ for $\lambda<\gamma$ with $x_{\alpha_{\lambda+1}}\in B(x,\eta)$.
	Let $\alpha_0$ be such that $x_{\alpha_0}\in B(x,\eta)\cap D_0$. Suppose now we have $\alpha_\lambda$. Let $\alpha_{\lambda+1}$ be such that $x_{\alpha_{\lambda+1}}\in B(x,\eta)\cap D_{\alpha_\lambda+1}$. Thus $\alpha_{\lambda+1}\geq\alpha_\lambda+1$. For a limit ordinal $\lambda$ put $\alpha_\lambda=\sup_{\tau<\lambda}\alpha_\tau$.
	The set $\{x_{\alpha_{\lambda+1}};\lambda<\gamma\}$ is an $\epsilon-$discrete subset of $B(x,\eta)$ with cardinality $\gamma$, a contradiction.
	
	Now we have that for every $\epsilon-$discrete set $E=\{x_\alpha;\alpha<\kappa\}$, $cf(\kappa)\leq\gamma$. To prove that $|E|\leq\gamma$ suppose first that $|E|=\aleph_{\alpha+1}$. Then we can take $\kappa=\omega_{\alpha+1}$ and so $|E|=|cf(\kappa)|\leq\gamma$. If $|E|=\aleph_\lambda$ for a limit ordinal $\lambda$, then $\aleph_\lambda=sup\{\aleph_{\alpha+1};\alpha<\lambda\}$. By the above we know that $\aleph_{\alpha+1}\leq\gamma$ for every $\alpha<\lambda$. Thus $\aleph_\lambda\leq\gamma$.
	
	
	Since $d(X)=sup\{|E|;E\text{ is an }\epsilon-\text{discrete set}\}$, we have that $d(X)\leq\gamma$.
\end{proof}	
\begin{cor}
	Let $(X,\rho)$ be a metric space such that each closed proper ball is separable. If $(CL(X),W_\rho)$ is normal, then $d(X)\leq\aleph_1$.
\end{cor}

\begin{lem}\label{factor}
	Let $(X,\rho)$ and $(Y,\delta)$ be metric spaces, let $k:X\to (0,\infty)$ be a function and let $f:X\to Y$ be a surjective map such that for every $x\in X$ and  $y\in Y$,  $\delta(y,f(x))=k(x)\rho(f^{-1}(y),x)$. $(CL(Y),W_\delta)$ can be embedded as a closed subset of $(CL(X),W_\rho)$.
\end{lem}
\begin{proof}
We will prove the statement in several steps.

1)	For every $x,x_0\in X$,  $\delta(f(x),f(x_0))=k(x_0)\rho(f^{-1}[f(x)],x_0)\leq k(x_0)\rho(x,x_0)$ and hence $f$ is continuous.

2)	For every $M\subset Y$ we have $f^{-1}(\overline M)=\overline{f^{-1}(M)}$: Since $f$ is continuous then $\overline{f^{-1}(M)}\subset f^{-1}(\overline M)$. Now take any $x\in f^{-1}(\overline M)$. There is a sequence $y_n\in M$ converging to $f(x)$. There is $x_n\in f^{-1}(y_n)\subset f^{-1}(M)$ such that $\rho(x_n,x)<\rho(f^{-1}(y_n),x)+\frac{1}{n}=\frac{1}{k(x)}\delta(y_n,f(x))+\frac{1}{n}\to 0$; i.e. $x\in\overline{f^{-1}(M)}$.

3)	We can define $g:CL(Y)\to CL(X)$ by $g(A)=f^{-1}(A)$. In the following steps we will prove that $g$ is the desired embedding.

4)	$g$ is injective; because $f$ is surjective.

5)	For every $x\in X$ and $A\in CL(Y)$ we have $k(x)\rho(g(A),x)=\delta(A,f(x))$: For every $\epsilon>0$ there is $y\in A$ such that $\delta(A,f(x))+\epsilon>\delta(y,f(x))=k(x)\rho(f^{-1}(y),x)\geq k(x)\rho(g(A),x)$. And for every $\epsilon>0$ there is $x_0\in g(A)$ (i.e. $f(x_0)\in A$) such that $k(x)\rho(g(A),x)+\epsilon>k(x)\rho(x_0,x)\geq\delta(f(x_0),f(x))\geq\delta(A,f(x))$.

6)	$g$ is continuous: Take a net $A_\lambda\in CL(Y)$ such that $A_\lambda\to A\in CL(Y)$. Let $x\in X$. $\rho(g(A_\lambda),x)=\frac{1}{k(x)}\delta(A_\lambda,f(x))\to\frac{1}{k(x)}\delta(A,f(x))=\rho(g(A),x)$ and so $g(A_\lambda)\to g(A)$.

7)	$g$ is closed: Let $\mathcal A$ be a closed subset of $CL(Y)$. Take a net $B_\lambda\in g(\mathcal A)$ such that $B_\lambda\to B\in CL(X)$. $B_\lambda=g(A_\lambda)$ where $A_\lambda\in\mathcal A$ and thus $\delta(A_\lambda,f(x))=k(x)\rho(g(A_\lambda),x)\to k(x)\rho(B,x)$ for every $x\in X$. Put $A=f(B)$. Now we will prove that $A$ is closed and $B=g(A)$. Naturally $B\subset f^{-1}(A)$. For the second inclusion suppose $x\in f^{-1}(A)$. Then $f(x)\in A=f(B)$, i.e. there is $x_0\in B$ satisfying $f(x_0)=f(x)$. Since $k(x)\rho(B,x)\leftarrow\delta(A_\lambda,f(x))=\delta(A_\lambda,f(x_0))\to k(x_0)\rho(B,x_0)$ then $\rho(B,x)=\frac{k(x_0)}{k(x)}\rho(B,x_0)=0$ and hence $x\in B$. Since $B$ is closed we have that $B=f^{-1}(A)=\overline{f^{-1}(A)}=f^{-1}(\overline A)$ and since $f$ is surjective then $A=\overline A$ and $B=g(A)$. It remains to prove that $B\in g(\mathcal A)$. For every $x\in X$ is $\delta(A_\lambda,f(x))\to k(x)\rho(B,x)=k(x)\rho(g(A),x)=\delta(A,f(x))$. Since $f(x)$ runs through all points of $Y$ we have that $A_\lambda\to A$, then $A\in\mathcal A$ and thus $B=g(A)\in g(\mathcal A)$.
\end{proof}

We can apply the above Lemma to the product of metric spaces. Note that many  frequently used definitions of the product metric can be written in the following form: Let $(X,\rho)$ and $(Y,\delta)$ be metric spaces and for $(x_1,y_1),(x_2,y_2)\in X\times Y$ put $\mu((x_1,y_1),(x_2,y_2))=\left\|(\rho(x_1,x_2),\delta(y_1,y_2))\right\|$; where $\left\|\cdot\right\|$ is a norm on $\mathbf R^2$ satisfying $\left\|(a,b)\right\|\geq\left\|(c,d)\right\|$ for $a\geq c\geq 0$ and $b\geq d\geq 0$.
\begin{cor}\label{product}
	Let $(X,\rho)$ and $(Y,\delta)$ be metric spaces. Then $(CL(X),W_\rho)$ and $(CL(Y),W_\delta)$ can be embedded as  closed subsets of $(CL(X\times Y),W_\mu)$, where $\mu$ is defined as above.
\end{cor}
\begin{proof}
	It is sufficient to prove it for $(CL(Y),W_\delta)$. For $(x,y)\in X\times Y$ define $f(x,y)=y$ and $k(x,y)=\left\|(0,1)\right\|^{-1}$. It is easy to verify the condition in Lemma \ref{factor} and we have the needed result.
\end{proof}
\begin{exa}
	Let $B(\mathfrak m)$ be the Baire space of the weight $\mathfrak m$ exactly as in \cite[4.2.12]{engelking} and $\sigma$ and $\rho$ be metrics described there. If $W_\rho$ or $W_\sigma$ is normal, then $\mathfrak m=\aleph_0$.
\end{exa}
\begin{proof}
	Let $(D(\mathfrak m),\mu)$ be a discrete metric space with the cardinality $\mathfrak m$ and $0-1$ metric. Then $B(\mathfrak m)=D(\mathfrak m)^{\aleph_0}$. For $\{x_i\},\{y_i\}\in B(\mathfrak m)$ we have
$$\sigma(\{x_i\},\{y_i\})=\sum_{i=1}^{\infty}\frac{\mu(x_i,y_i)}{2^i}$$ and
$$\rho(\{x_i\},\{y_i\})=\left\{
\begin{array}{ll}1/k,&if\ x_k\not=y_k\ and\ x_i=y_i\ for\ i<k\\0,&if\ x_i=y_i\ for\ all\ i.
\end{array}\right.$$
Observe that $B(\mathfrak m)=D(\mathfrak m)\times B(\mathfrak m)$ and $\sigma(\{x_i;i\geq 1\},\{y_i;i\geq 1\})=\mu(x_1,y_1)/2+\sigma(\{x_{i+1};i\geq 1\},\{y_{i+1};i\geq 1\})/2$. So  by \ref{product} we have that $(CL(D(\mathfrak m)),W_\mu)$ can be embedded as a closed subset of $(CL(B(\mathfrak m)),W_\sigma)$ and the rest follows. To show that $(CL(D(\mathfrak m)),W_\mu)$ can be embedded as a closed subset of $(CL(B(\mathfrak m)),W_\rho)$ observe that $D(\mathfrak m)$ is isometrically isomorphic to $Y\subset B(\mathfrak m)$ consisting of all constant sequences. The rest follows from \ref{subset}.
\end{proof}
\begin{cor}\label{parallel}
	Let $(Y,\delta)$ be a metric space with $0-1$ metric. Let $X=\cup\{X_y;y\in Y\}$ where $X_y$ are mutually disjoint and $\rho$ is a metric on $X$ such that for $x_y\in X_y$, $x_z\in X_z$, $y\not=z$  $\rho(x_y,x_z)=1$.  Then $(CL(Y),W_\delta)$ can be embedded as  a closed subspace of $(CL(X),W_\rho)$.
\end{cor}
\begin{proof}
	In Lemma \ref{factor} just check the condition for $f,k$ defined by $k(x)=1$ and $f(x_y)=y$ for $x_y\in X_y$.
\end{proof}
\begin{prop}
	Let $(X,\rho)$ and $(Y,\delta)$ be as in Corollary \ref{parallel}. If $(CL(X),W_\rho)$ is normal and $X_y$ is separable for all $y\in Y$, then $X$ is separable.
\end{prop}
\begin{proof}
	By Corollary \ref{parallel}  $(CL(Y),W_\delta)$ is embedded as a closed subset of $(CL(X),W_\rho)$. Thus $(CL(Y),W_\delta)$ is normal, hence it is countable. $X$ is a countable union of separable spaces, so it is separable.
\end{proof}
	The metric space in \cite[Example 6]{barcos} is of this type. Also a metric space with  $M-$metric is of this type. So we have the following result.
\begin{prop}
	Let $(X,\rho_M)$ be a metric space with $M-$metric. If $(CL(X),\rho_M)$ is normal, then $X$ is countable.
\end{prop}
\begin{thm}\label{pol}
	Let $(X,\rho)$ be a metric space and $Y$ be a set of points of $X$ with a compact neighborhood. If $Y$ is separable and $(CL(X),W_\rho)$ is normal, then $X$ is separable.
\end{thm}
\begin{proof}
	This proof uses an idea of \cite{chapol}. Suppose that $X$ is not separable, then there is $\epsilon>0$ and an $\epsilon-$discrete set $T\subset X\setminus Y$, with $|T|=\aleph_1$. We want to prove that $\omega^{\aleph_1}$ is a closed subset of $(CL(X),W_\rho)$, which is hence not normal. For $t\in T$ put $B_t=B(t,\epsilon/5)$ and $S_t=S(t,\epsilon/4)$. Since $B_t$ is not compact it contains a countable closed discrete set $\{x_{t,n};n\in\omega\}$. Let $\omega^T$ be a space of functions $u:T\to\omega$ with the pointwise topology; i.e. $\omega^T=\omega^{\aleph_1}$. Define $g:\omega^T\to CL(X)$ by $g(u)=\{x_{t,u(t)};t\in T\}\cup(X\setminus\bigcup_{t\in T}S_t)$. The function $g$ is obviously injective and the set $\mathcal F = g(\omega^T)$ is closed in $(CL(X),W_\rho)$. We can use the same idea as in \cite{chapol} to prove that $g:\omega^T\to \mathcal F$ is a homeomorphism, where $\mathcal F$ is considered with the relative Wijsman topology.

\end{proof}


We have the following corollaries:
\begin{cor}
	Let $(X,\rho)$ be a linear metric space. If  $(CL(X),W_\rho)$ is normal, then $X$ is separable.
\end{cor}
\begin{proof}
	Suppose $X$ is not separable, then it is infinitely dimensional and hence no point has a compact neighborhood. By Theorem \ref{pol}, $(CL(X),W_\rho)$ cannot be normal.
\end{proof}
\begin{cor}
	Let $(X,\left\|.\right\|)$ be a normed linear space and $\rho$ be a metric generated by $\left\|.\right\|$.
If  $(CL(X),W_\rho)$ is normal, then $X$ is separable.
\end{cor}
\bigskip
Notice that if $X$ is a locally compact metrizable  space, then there is a compatible metric $\eta$ such that the normality of $(CL(X),W_\eta)$ implies the separability of $X$. In fact, by \cite{beer} there is a compatible metric $\eta$ with nice closed balls. Thus by Corollary 5.8 the normality of $(CL(X),W_\eta)$ implies the separability of $X$.


\end{document}